\documentclass[a4paper,12pt]{amsart}

\usepackage{amsfonts}
\usepackage{amsmath}
\usepackage{amssymb}
\usepackage{amsthm}
\usepackage{color}
\usepackage{multirow}
\usepackage{enumerate}

\usepackage[
    colorlinks,
    linkcolor={black},
    menucolor={black},
    citecolor={blue},
    urlcolor={blue}
]{hyperref}

\newtheorem{theorem}{Theorem}[section]

\newtheorem{proposition}[theorem]{Proposition}

\newenvironment{multilinecol}[1]{\begin{minipage}{#1}\mbox{}\vspace*{-0.3em}\\}{\\
\vspace*{-0.7em}
\end{minipage}}

\def\mm{{\mathfrak m}}

\begin{document}


\title{The far-flung Gorenstein numerical semigroup rings of type 4}
\author{Teodor I. Grigorescu}
\address{Teodor I. Grigorescu, Faculty of Mathematics and Computer Science, University of Bucharest, Str. Academiei 14, Bucharest 010014, Romania}
\email{teodor$\_999$@yahoo.com}

\maketitle

\begin{abstract}

We classify the far-flung Gorenstein numerical semigroup rings of\\ type 4.

\end{abstract}

\section*{Introduction}

The far-flung Gorenstein property was introduced in \cite{f} as a means to understand the space of Cohen-Macaulay rings around the Gorenstein property. 

Let $(R,\mm)$ be a Cohen-Macaulay local ring with a canonical module $\omega_R$. In recent years, properties of rings have been investigated in relation to the trace ideal of $\omega_R$ (\cite{a}, \cite{c}, \cite{d}, \cite{e}, \cite{Herzog-Kumashiro}, \cite{f}, \cite{g}). For a ring $R$ and for an $R$-module $M$, the trace ideal of $M$ is
$$\mbox{tr}(M):=\sum_{\varphi\in\mbox{Hom}_R(M,R)}\mbox{Im}(\varphi).$$
It is known that, for any $P\in \mbox{Spec}(R)$, $R_{P}$ is not Gorenstein if and only if $\mbox{tr}(\omega_R)\subseteq P$ (\cite[Lemma 2.1]{d}). In particular, $R$ is a Gorenstein ring precisely when $\mbox{tr}(\omega_R)=R$. Therefore, when $\mbox{tr}(\omega_R)$ is large, the ring $R$ is in some sense close to being Gorenstein, and the other way around. For instance, in \cite[Definition 2.2]{d}, the ring $R$ is called {\it nearly Gorenstein} if $\mbox{tr}(\omega_R)\supseteq \mm$.

We further assume $R$ is one-dimensional. It is proved in \cite[Proposition A.1]{e} that, if $R$ is a domain, then $\mbox{tr}(\omega_R)\supseteq R:\overline{R}$, where $\overline{R}$ denotes the integral closure of $R$ and the colon is taken in the total ring of fractions $Q(R)$. Motivated by this, in \cite{f}, Herzog, Kumashiro and Stamate study the following class of rings:
\vspace{0.3cm}

\noindent{\bf Definition} (\cite{f}){\bf .} Let $R$ be a one-dimensional Cohen-Macaulay local ring with a canonical module $\omega_R$. The ring $R$ is called {\it far-flung Gorenstein} if $\mbox{tr}(\omega_R)=R:\overline{R}$.
\vspace{0.3cm}

In \cite[pp. 628]{f}, the authors give some technical conditions on $R$ under which they obtain some nice properties of the class of far-flung Gorenstein rings. One condition is the existence of an $R$-module $C\cong \omega_R$ with
\begin{equation}
\label{conditie}
R\subseteq C\subseteq \overline{R}.
\end{equation}
The latter holds, for example, when $R$ is a domain with infinite residue field. Then, under the said technical conditions, the ring $R$ is far-flung Gorenstein if and only if $C^2=\overline{R}$, see \cite[Theorem 2.5]{f}. One large class of rings which fulfill the conditions in \cite[pp. 628]{f} are the numerical semigroup rings. A numerical semigroup $H$ is a submonoid of $(\mathbb N,+)$ such that $\mathbb N\setminus H$ is finite. If $K$ is any field, the numerical semigroup ring $K[[H]]$ is defined as the subring $K[[H]]:=K[[t^h : h\in H]]$ of the power series ring $K[[t]]$. 

The (Cohen-Macaulay) type is an important invariant of a Cohen-Macaulay local ring. Crucially, the class of Gorenstein local rings consists precisely of all Cohen-Macaulay local rings of type 1. We refer to \cite{Bruns-Herzog} for more background on this topic.

For far-flung Gorenstein rings, in \cite{f}, the authors provide bounds for the multiplicity of the ring in terms of its type. Thus it is proved in \cite[Example 2.9]{f} that $K[[H]]$ is a far-flung Gorenstein ring of type 2 if and only if $H=\langle 3,3m+1,3m+2\rangle$, with $m\geq 1$. Then, in \cite{f}, the authors characterize when $K[[H]]$ is far-flung Gorenstein in terms of the combinatorics of $H$. In \cite{Froberg}, Fr\"{o}berg, Gottlieb and H\"{a}ggkvist define the type of the numerical semigroup $H$ as the cardinality of the set of its pseudo-Frobenius numbers
$$PF(H):=\{x\in \mathbb Z\setminus H : x+h\in H, ~\forall ~h\in H\setminus \{0\}\}.$$
The largest element of $PF(H)$ is the Frobenius number of $H$, denoted $F(H)$. It is known (\cite[Theorem 3.1]{k}) that the type of $K[[H]]$ equals the type of $H$. Then, according to \cite{f}, $K[[H]]$ is far-flung Gorenstein if and only if
$$\{0,...,e(H)-1\}\subseteq \{2F(H)-\alpha-\beta : \alpha,\beta\in PF(H)\},$$
where $e(H):=\mbox{min}~(H\setminus\{0\})$ is the multiplicity of the numerical semigroup $H$ and, at the same time, of the local ring $K[[H]]$. It is also proved in \cite{f} that, if $K[[H]]$ is a far-flung Gorenstein ring of type $r$, the following bounds exist for the multiplicity of $H$:
\begin{equation}
\label{marg}
r+1\leq e(H)\leq \bar{n}(r),
\end{equation}
where $\bar{n}(r)$ is the solution of the Rohrbach problem that we recall in Section \ref{prelim}. The lower bound $e(H)=r+1$ corresponds to the case when $H$ is of minimal multiplicity. Using these results, it is obtained in \cite[Theorem 6.4]{f} the classification of the far-flung Gorenstein numerical semigroup rings of type 3 and not of minimal multiplicity, see also Theorem \ref{completare}. We adapt these tools to tackle the type 4 case.

The purpose of this note is to describe the far-flung Gorenstein numerical semigroup rings of type 4, see Theorem \ref{clasificare}. The main result is the classification of the far-flung Gorenstein numerical semigroup rings of type 4 presented in Table 1. As noted in Table 1, the 40 families depend on one or two parameters and the embedding dimension is 4, 5, 6 or 7. The generators listed in Table 1 are minimal except for the values of the parameters given in Table 2.

In Section \ref{prelim}, we recall how algebraic properties of the ring $K[[H]]$, especially the far-flung Gorenstein property, can be translated into properties of the numerical semigroup $H$. The key observation is that the generators of the canonical $K[[H]]$-module $C$ satisfying condition \eqref{conditie} mentioned above can be obtained from the pseudo-Frobenius numbers of $H$. Then we complete the classification of the far-flung Gorenstein numerical semigroup rings of type 3 obtained in \cite{f} by including the case of minimal multiplicity, see Theorem \ref{completare}.

In Section \ref{type4}, we classify the far-flung Gorenstein numerical semigroup rings of type 4 starting from the inequality $5\leq e(H)\leq 9$ coming from \eqref{marg}, see Proposition \ref{margini}. For each value in this range, in view of Theorem \ref{caract}, there are several possibilities for the canonical module $C$. Knowledge of $C$ gives information about the relative distance among the pseudo-Frobenius numbers. This results in many restrictions that need to be studied. Some of these restrictions do not produce far-flung Gorenstein numerical semigroup rings, as, for example, the subcase corresponding to Equation \eqref{eq8} in the proof of Theorem \ref{clasificare}. For each possible value of $e(H)$, we give details for one or two relevant subcases, as the technique is similar for the rest.

\section{Preliminaries}
\label{prelim}

A numerical semigroup $H$ is a submonoid of ($\mathbb N$,+) such that $\mathbb N\setminus H$ is finite.

The {\it multiplicity} of a numerical semigroup $H$ is defined as $e(H):=\min~(H\setminus \{0\})$. The {\it embedding dimension} of $H$, denoted $\mu(H)$, is the cardinality of the unique minimal generating set of $H$. The {\it Frobenius number} of $H$ is $F(H):=\max~(\mathbb Z\setminus H)$. The set of the {\it pseudo-Frobenius numbers} of $H$ is 
$$PF(H):=\{x\in \mathbb Z\setminus H : x+h\in H, ~\forall ~h\in H\setminus \{0\}\}$$
and its cardinality is called the {\it type} of $H$, denoted $r(H)$. Note that $F(H)=\max~PF(H)$.

The Ap\'{e}ry set of $H$ with respect to a nonzero integer $n\in H$ is
$${\rm Ap}(H,n):=\{h\in H : h-n\notin H\}.$$

It always holds that $e(H)\geq \mu(H)$. The numerical semigroup $H$ is said to be {\it of minimal multiplicity} if $e(H)=\mu(H)$. In this case, the pseudo-Frobenius numbers of $H$ are computed as follows.

\begin{proposition}[{\cite[Proposition 2.20 and Proposition 3.1]{i}}]
\label{PF}

Let $H\subsetneq \mathbb N$ be a numerical semigroup minimally generated by $a_1<a_2<...<a_v$ which is of minimal multiplicity. Then $PF(H)=\{a_2-a_1,a_3-a_1,...,a_v-a_1\}$.

\end{proposition}

Let $H$ be any numerical semigroup. $H$ induces on $\mathbb Z$ the partial ordering:
$$a\leq_Hb \Longleftrightarrow b-a\in H.$$
It is proved in \cite{i} that the pseudo-Frobenius numbers of $H$ are the maximal elements of $\mathbb Z\setminus H$ with respect to the partial ordering $\leq_H$.

\begin{proposition}[{\cite[Proposition 2.19 and Proposition 2.20]{i}}]
\label{maximale}

Let $H$ be a numerical semigroup. Then
$$PF(H)={\rm Maximals}_{\leq_H}(\mathbb Z\setminus H)={\rm Maximals}_{\leq_H}({\rm Ap}(H,e(H))-e(H)).$$

\end{proposition}

We refer to \cite{i} for more background on numerical semigroups.
\\

Let $K$ be any field. The semigroup ring associated to $H$ is $K[[H]]:=K[[t^h : h\in H]]$, a subring of the power series ring $K[[t]]$. Then $\overline{K[[H]]}=K[[t]]$.

\begin{proposition}[{\cite[Exercise 21.11]{b}}]
\label{C}

Let $K[[H]]$ be a numerical semigroup ring. Then $C=(t^{F(H)-\alpha} : \alpha\in PF(H))$ is a canonical module of $K[[H]]$ which satisfies \eqref{conditie}.

\end{proposition}

For the rest of the paper, we use the canonical module $C$ described in Proposition \ref{C}.
\\

In Theorem \ref{caract}, we collect characterizations of the far-flung Gorenstein property for numerical semigroup rings, as proved in \cite{f}.

\begin{theorem}[{\cite[Theorem 2.5, Theorem 5.1 and Proposition 6.1]{f}}]
\label{caract}

Let $H\subsetneq \mathbb N$ be a numerical semigroup. Then the following statements are equivalent:
\begin{enumerate}[\rm (i)]
\item $K[[H]]$ is a far-flung Gorenstein ring;

\item $C^2=K[[t]]$;

\item $\{0,...,e(H)-1\}\subseteq \{2F(H)-\alpha-\beta : \alpha,\beta\in PF(H)\}$;

\item There exist $n_1<...<n_{r(H)}\in \mathbb N$ such that $C=(t^{n_1},...,t^{n_{r(H)}})K[[H]]$ and\\ $\{0,...,e(H)-1\}\subseteq \{n_i+n_j : 1\leq i\leq j\leq r(H)\}$.
\end{enumerate}

If these hold, we may assume in {\rm (iv)} that $n_1=0$ and $n_2=1$.

\end{theorem}

In \cite{f}, the solutions of the Rohrbach problem were used to obtain an upper bound for the multiplicity of a far-flung Gorenstein numerical semigroup ring.
\\

{\noindent \bf The Rohrbach problem} (\cite{h}, \cite{j}){\bf .} For a finite set of non-negative integers $A$, let $n(A)$ be the integer such that $0,1,...,n(A)-1\in A+A$ and $n(A)\notin A+A$. If $0\notin A$, we consider that $n(A)=-1$.

The Rohrbach problem consists of determining the integers
$$\bar{n}(r):=\max~\{n(A) : |A|=r\},$$
where $r\geq 1$.
The solution of the Rohrbach problem is known for $r\leq 25$ (\cite{j}), but no formula for $\bar{n}(r)$ is available so far. For example, $\bar{n}(1)=1$, $\bar{n}(2)=3$, $\bar{n}(3)=5$ and $\bar{n}(4)=9$.
\\

The following proposition provides a lower bound and an upper bound for the\\ multiplicity of a far-flung Gorenstein numerical semigroup ring in terms of its type.

\begin{proposition}[{\cite[Corollary 2.23 and Corollary 3.2]{i}, \cite[Corollary 5.3]{f}}]
\label{margini}

Let $K[[H]]$ be a far-flung Gorenstein numerical semigroup ring of type $r\geq 2$. Then 
$$r+1\leq e(H)\leq \bar{n}(r),$$
where $\bar{n}(r)$ is the solution of the Rohrbach problem.

In addition, $e(H)=r+1$ if and only if $H$ is of minimal multiplicity.

\end{proposition}

The classification of the far-flung Gorenstein numerical semigroup rings of type 3 is essentially given in \cite{f}. We add to it the case when $H$ is of minimal multiplicity.

\begin{theorem}[{\cite[Theorem 6.4]{f}}]
\label{completare}

Let $K[[H]]$ be a numerical semigroup ring of {type 3}. Then $K[[H]]$ is far-flung Gorenstein if and only if $H$ is in one of the following families of numerical semigroups:
\begin{enumerate}[\rm 1.]
\item $H=\langle 4,4m+1,4m+2,4m+3\rangle$, with $m\geq 1$;

\item $H=\langle 4,4m+3,4m+5,4m+6\rangle$, with $m\geq 1$;

\item $H=\langle 5,5m+4,10m+6,10m+7\rangle$, with $m\geq 1$;

\item $H=\langle 5,5m+1,10m+3,10m+4\rangle$, with $m\geq 1$;

\item $H=\langle 5,5m+2,10m+1,10m+3\rangle$, with $m\geq 1$;

\item $H=\langle 5,5m+3,10m+4,10m+7\rangle$, with $m\geq 1$.
\end{enumerate}

\end{theorem}

\begin{proof}
It is proved in \cite{f} that $K[[H]]$ is a far-flung Gorenstein ring of type 3 not of minimal multiplicity if and only if $H$ is in one of the families numbered 3 to 6 in the list.

Assume $K[[H]]$ is a far-flung Gorenstein ring of type 3 which is of minimal multiplicity. Then, by Proposition \ref{margini}, $e(H)=4$. By Theorem \ref{caract}, there exists $n\geq 2$ such that the canonical module $C$ described in Proposition \ref{C} is of the form $C=(1,t,t^n)$ and $\{0,1,2,3\}\subseteq \{0,1,2,n,n+1,2n\}$. Hence $C=(1,t,t^2)$ or $C=(1,t,t^3)$.

Assume $C=(1,t,t^2)$. By Proposition \ref{C}, $PF(H)=\{F(H)-2,F(H)-1,F(H)\}$. Denoting $a=F(H)-2$, we get $PF(H)=\{a,a+1,a+2\}$, with $a\geq 1$. Since the elements of $PF(H)$ have distinct nonzero values modulo 4, we get $PF(H)=\{4n+1,4n+2,4n+3\}$, with $n\geq 0$. Then $H=\langle 4,4n+5,4n+6,4n+7\rangle$, with $n\geq 0$. Conversely, it is easy to check that, for such an $H$, $C=(1,t,t^2)$.

When $C=(1,t,t^3)$, arguing as before, $PF(H)=\{a,a+2,a+3\}$, with $a\geq 1$. Since the elements of $PF(H)$ have distinct nonzero values modulo 4, we get $PF(H)=\{4n+3,4n+5,4n+6\}$, with $n\geq 0$. Then $H=\langle 4,4n+7,4n+9,4n+10\rangle$, with $n\geq 0$. Conversely, it is easy to check that, for such an $H$, $C=(1,t,t^3)$.
\end{proof}

\section{The far-flung Gorenstein numerical semigroup rings of type 4}
\label{type4}

In the following theorem, we classify all 40 disjoint families of far-flung Gorenstein numerical semigroup rings of type 4, as presented in Table 1.

\begin{theorem}
\label{clasificare}

Let $K[[H]]$ be a numerical semigroup ring of {type 4}. Then $K[[H]]$ is far-flung Gorenstein if and only if $H$ is in one of the families of numerical semigroups listed in Table 1.

\end{theorem}

\begin{table}
\caption{The far-flung Gorenstein numerical semigroups of type 4}
{\tiny
\begin{tabular}{|c|c|c|c|}
\hline
No. & $H$ & $PF(H)$ & $C$\\ \hline
1 & \begin{multilinecol}{6cm} $\langle 5,5m+1,5m+5x-3,5m+5x-2,5m+5x-1\rangle$, $1\leq x\leq m$ \end{multilinecol} & $\{5m-4,5m+5x-8,5m+5x-7,5m+5x-6\}$ & $(1,t,t^2,t^{5x-2})$\\ \hline
2 & \begin{multilinecol}{6cm} $\langle 5,5m+4,5m+5x+1,5m+5x+2,5m+5x+3\rangle$, $1\leq x\leq m$ \end{multilinecol} & $\{5m-1,5m+5x-4,5m+5x-3,5m+5x-2\}$ & $(1,t,t^2,t^{5x-1})$\\ \hline
3 & \begin{multilinecol}{6cm} $\langle 5,5m+2,5m+5x+1,5m+5x+3,5m+5x+4\rangle$, $1\leq x\leq m-1$ \end{multilinecol} & $\{5m-3,5m+5x-4,5m+5x-2,5m+5x-1\}$ & $(1,t,t^3,t^{5x+2})$\\ \hline
4 & \begin{multilinecol}{6cm} $\langle 5,5m+3,5m+5x-1,5m+5x+1,5m+5x+2\rangle$, $1\leq x\leq m$ \end{multilinecol} & $\{5m-2,5m+5x-6,5m+5x-4,5m+5x-3\}$ & $(1,t,t^3,t^{5x-1})$\\ \hline
5 & $\langle 6,6m+5,12m+7,12m+8,12m+9\rangle$, $m\geq 1$ & $\{12m+1,12m+2,12m+3,12m+4\}$ & $(1,t,t^2,t^3)$\\ \hline
6 & $\langle 6,6m+1,12m+3,12m+4,12m+5\rangle$, $m\geq 1$ & $\{12m-4,12m-3,12m-2,12m-1\}$ & $(1,t,t^2,t^3)$\\ \hline
7 & $\langle 6,6m+2,12m+1,12m+3,12m+5\rangle$, $m\geq 1$ & $\{12m-5,12m-3,12m-2,12m-1\}$ & $(1,t,t^2,t^4)$\\ \hline
8 & $\langle 6,6m+4,12m+5,12m+7,12m+9\rangle$, $m\geq 1$ & $\{12m-1,12m+1,12m+2,12m+3\}$ & $(1,t,t^2,t^4)$\\ \hline
9 & $\langle 6,6m+5,12m+13,12m+14,12m+15\rangle$, $m\geq 1$ & $\{12m+4,12m+7,12m+8,12m+9\}$ & $(1,t,t^2,t^5)$\\ \hline
10 & $\langle 6,6m+1,12m-1,12m+3,12m+4\rangle$, $m\geq 1$ & $\{12m-7,12m-4,12m-3,12m-2\}$ & $(1,t,t^2,t^5)$\\ \hline
11 & $\langle 6,6m+4,12m+3,12m+5,12m+7\rangle$, $m\geq 1$ & $\{12m-3,12m-1,12m+1,12m+2\}$ & $(1,t,t^3,t^5)$\\ \hline
12 & $\langle 6,6m+2,12m-1,12m+1,12m+3\rangle$, $m\geq 1$ & $\{12m-7,12m-5,12m-3,12m-2\}$ & $(1,t,t^3,t^5)$\\ \hline
13 & \begin{multilinecol}{6cm} $\langle 7,7m+5,7x+6,7m+7x+8,7m+7x+9,7m+7x+10\rangle$, $m,x\geq 1$ and $\frac{m-1}{2}\leq x\leq m$ \end{multilinecol} & $\{7m+7x+1,7m+7x+2,7m+7x+3,7m+7x+4\}$ & $(1,t,t^2,t^3)$\\ \hline
14 & \begin{multilinecol}{6cm} $\langle 7,7m+1,7x+2,7m+7x+4,7m+7x+5,7m+7x+6\rangle$, $1\leq m\leq x\leq 2m$ \end{multilinecol} & $\{7m+7x-4,7m+7x-3,7m+7x-2,7m+7x-1\}$ & $(1,t,t^2,t^3)$\\ \hline
15 & $\langle 7,7m+1,7m+3,14m+5\rangle$, $m\geq 1$ & $\{14m-5,14m-3,14m-2,14m-1\}$ & $(1,t,t^2,t^4)$\\ \hline
16 & $\langle 7,7m+4,7m+5,14m+6\rangle$, $m\geq 1$ & $\{14m-1,14m+1,14m+2,14m+3\}$ & $(1,t,t^2,t^4)$\\ \hline
17 & $\langle 7,7m+2,7m+3,14m+1\rangle$, $m\geq 1$ & $\{14m-6,14m-3,14m-2,14m-1\}$ & $(1,t,t^2,t^5)$\\ \hline
18 & \begin{multilinecol}{6cm} $\langle 7,7m+4,7x+6,7m+7x+5,7m+7x+8,7m+7x+9\rangle$, $m,x\geq 1$ and $m-1\leq x\leq m$ \end{multilinecol} & $\{7m+7x-2,7m+7x+1,7m+7x+2,7m+7x+3\}$ & $(1,t,t^2,t^5)$\\ \hline
19 & \begin{multilinecol}{6cm} $\langle 7,7m+3,7x+6,7m+7x+8,7m+7x+11,7m+7x+12\rangle$, $1\leq m\leq x\leq 2m$ \end{multilinecol} & $\{7m+7x+1,7m+7x+2,7m+7x+4,7m+7x+5\}$ & $(1,t,t^3,t^4)$\\ \hline
20 & \begin{multilinecol}{6cm} $\langle 7,7m+1,7x+4,7m+7x+2,7m+7x+3,7m+7x+6\rangle$, $m,x\geq 1$ and $\frac{m-1}{2}\leq x\leq m$ \end{multilinecol} & $\{7m+7x-5,7m+7x-4,7m+7x-2,7m+7x-1\}$ & $(1,t,t^3,t^4)$\\ \hline
21 & $\langle 7,7m+3,7m+4,14m+5,14m+9\rangle$, $m\geq 1$ & $\{14m-2,14m-1,14m+1,14m+2\}$ & $(1,t,t^3,t^4)$\\ \hline
22 & \begin{multilinecol}{6cm} $\langle 7,7m+2,7x+4,7m+7x+1,7m+7x+3,7m+7x+5\rangle$, $m,x\geq 1$ and $m-1\leq x\leq 2m$ \end{multilinecol} & $\{7m+7x-6,7m+7x-4,7m+7x-2,7m+7x-1\}$ & $(1,t,t^3,t^5)$\\ \hline
23 & \begin{multilinecol}{6cm} $\langle 7,7m+3,7x+5,7m+7x+4,7m+7x+6,7m+7x+9\rangle$, $m,x\geq 1$ and $\frac{m-1}{2}\leq x\leq m$ \end{multilinecol} & $\{7m+7x-3,7m+7x-1,7m+7x+1,7m+7x+2\}$ & $(1,t,t^3,t^5)$\\ \hline
24 & \begin{multilinecol}{6cm} $\langle 8,8m+1,8m+3,8x+4,8m+8x+2,8m+8x+6\rangle$, $1\leq x\leq m$ \end{multilinecol} & $\{8m+8x-6,8m+8x-3,8m+8x-2,8m+8x-1\}$ & $(1,t,t^2,t^5)$\\ \hline
25 & \begin{multilinecol}{6cm} $\langle 8,8m+4,8x+5,8x+7,8m+8x+6,8m+8x+10\rangle$, $1\leq m\leq x$ \end{multilinecol} & $\{8m+8x-2,8m+8x+1,8m+8x+2,8m+8x+3\}$ & $(1,t,t^2,t^5)$\\ \hline
26 & \begin{multilinecol}{6cm} $\langle 8,8m+3,8x+6,8x+7,8m+8x+12,8m+8x+13\rangle$, $1\leq m\leq x\leq 2m$ \end{multilinecol} & $\{8m+8x+1,8m+8x+2,8m+8x+4,8m+8x+5\}$ & $(1,t,t^3,t^4)$\\ \hline
27 & \begin{multilinecol}{6cm} $\langle 8,8m+3,8m+6,8x+7,8m+8x+9,8m+8x+12\rangle$, $m,x\geq 1$ and $\frac{m-1}{2}\leq x\leq m-1$ \end{multilinecol} & $\{8m+8x+1,8m+8x+2,8m+8x+4,8m+8x+5\}$ & $(1,t,t^3,t^4)$\\ \hline
28 & \begin{multilinecol}{6cm} $\langle 8,8m-1,8m+1,8x+4,8m+8x+2,8m+8x+6\rangle$, $1\leq x\leq m-1$ \end{multilinecol} & $\{8m+8x-6,8m+8x-5,8m+8x-3,8m+8x-2\}$ & $(1,t,t^3,t^4)$\\ \hline
29 & \begin{multilinecol}{6cm} $\langle 8,8m+1,8x+2,8x+5,8m+8x+4,8m+8x+7\rangle$, $1\leq m\leq x\leq 2m$ \end{multilinecol} & $\{8m+8x-5,8m+8x-4,8m+8x-2,8m+8x-1\}$ & $(1,t,t^3,t^4)$\\ \hline
30 & \begin{multilinecol}{6cm} $\langle 8,8m+1,8m+2,8x+5,8m+8x+3,8m+8x+4\rangle$, $m,x\geq 1$ and $\frac{m-1}{2}\leq x\leq m-1$ \end{multilinecol} & $\{8m+8x-5,8m+8x-4,8m+8x-2,8m+8x-1\}$ & $(1,t,t^3,t^4)$\\ \hline
31 & \begin{multilinecol}{6cm} $\langle 8,8m+3,8x+4,8m+5,8m+8x+6,8m+8x+10\rangle$, $1\leq x\leq m$ \end{multilinecol} & $\{8m+8x-2,8m+8x-1,8m+8x+1,8m+8x+2\}$ & $(1,t,t^3,t^4)$\\ \hline
32 & \begin{multilinecol}{6cm} $\langle 9,9m+3,9m+6,9x+7,9x+8\rangle$, $1\leq m\leq x$ \end{multilinecol} & $\{9m+9x+1,9m+9x+2,9m+9x+4,9m+9x+5\}$ & $(1,t,t^3,t^4)$\\ \hline
33 & \begin{multilinecol}{6cm} $\langle 9,9m+3,9m+6,9x+7,18x-9m+8,18x+10,\\18x+13\rangle$, $m\geq 3$ and $\frac{2m-1}{3}\leq x\leq m-1$ \end{multilinecol} & $\{18x+1,18x+2,18x+4,18x+5\}$ & $(1,t,t^3,t^4)$\\ \hline
34 & \begin{multilinecol}{6cm} $\langle 9,9m+1,18m-9x-5,18m-9x-2,9x+8,18m+5\rangle$, $m,x\geq 1$ and $\frac{2m-2}{3}\leq x\leq m-1$ \end{multilinecol} & $\{18m-7,18m-6,18m-4,18m-3\}$ & $(1,t,t^3,t^4)$\\ \hline
35 & \begin{multilinecol}{6cm} $\langle 9,9m-2,9m-1,9m+1,9x+4,9m+9x+6\rangle$,\\ $m,x\geq 1$ and $\frac{m-1}{2}\leq x\leq m-1$ \end{multilinecol} & $\{9m+9x-7,9m+9x-6,9m+9x-4,9m+9x-3\}$ & $(1,t,t^3,t^4)$\\ \hline
36 & \begin{multilinecol}{6cm} $\langle 9,9m+1,18x-9m+11,18x-9m+14,9x+8,\\18x+13\rangle$, $1\leq m\leq x\leq \frac{3m-1}{2}$ \end{multilinecol} & $\{18x+3,18x+4,18x+6,18x+7\}$ & $(1,t,t^3,t^4)$\\ \hline
37 & \begin{multilinecol}{6cm} $\langle 9,9m-1,9m+1,9m+2,9x+5,9m+9x+3\rangle$,\\ $m,x\geq 1$ and $\frac{m-1}{2}\leq x\leq m-1$ \end{multilinecol} & $\{9m+9x-6,9m+9x-5,9m+9x-3,9m+9x-2\}$ & $(1,t,t^3,t^4)$\\ \hline
38 & \begin{multilinecol}{6cm} $\langle 9,9m+1,9x+2,18x-9m+3,18x-9m+6,18x+5,18x+8\rangle$, $1\leq m\leq x\leq 2m$ \end{multilinecol} & $\{18x-5,18x-4,18x-2,18x-1\}$ & $(1,t,t^3,t^4)$\\ \hline
39 & \begin{multilinecol}{6cm} $\langle 9,9x+3,9x+6,9m+1,9m+2\rangle$, $1\leq x\leq m-1$ \end{multilinecol} & $\{9m+9x-5,9m+9x-4,9m+9x-2,9m+9x-1\}$ & $(1,t,t^3,t^4)$\\ \hline
40 & \begin{multilinecol}{6cm} $\langle 9,9m+3,9x+4,9x+5,9m+6\rangle$, $1\leq m\leq x$ \end{multilinecol} & $\{9m+9x-2,9m+9x-1,9m+9x+1,9m+9x+2\}$ & $(1,t,t^3,t^4)$\\ \hline
\end{tabular}}
\end{table}

\begin{proof}
Assume $K[[H]]$ is a far-flung Gorenstein ring. By Proposition \ref{margini}, $5\leq e(H)\leq 9$. For each of these values, in view of Theorem \ref{caract}, there are several possibilities for the canonical module $C$ described in Proposition \ref{C}. For each $C$, arguing modulo $e(H)$, we note that there are $e(H)-4$ possible sets of pseudo-Frobenius numbers. This results in many restrictions that need to be studied. Most of these restrictions produce far-flung Gorenstein numerical semigroups, but not all of them (see the discussion below in the case $e(H)=8$). The technique is similar from one subcase to another. For brevity, we will give details for one positive subcase for each possible value of $e(H)$. In the case $e(H)=8$, we also describe a candidate set of pseudo-Frobenius numbers which does not occur.
\\

I. We will assume that $e(H)=5$.

Since $e(H)=5=r+1$, by Proposition \ref{margini}, $H$ is of minimal multiplicity.

By Theorem \ref{caract}, there exist $2\leq n_3<n_4$ such that $C=(1,t,t^{n_3},t^{n_4})$ and \begin{equation}
\label{eq2}
\{0,1,2,3,4\}\subseteq \{0,1,2,n_3,n_3+1,2n_3,n_4,n_4+1,n_3+n_4,2n_4\}.
\end{equation}
Hence $2\leq n_3\leq 3$ and any $n_4>n_3$ satisfies \eqref{eq2}.

If $n_3=2$, then $C=(1,t,t^2,t^{n_4})$. By Proposition \ref{C}, $PF(H)=\{F(H)-n_4,F(H)-2,F(H)-1,F(H)\}$. Denoting $a=F(H)-n_4$, we get $PF(H)=\{a,a+n_4-2,a+n_4-1,a+n_4\}$, with $a\geq 1$. Since the elements of $PF(H)$ have distinct nonzero values modulo 5, we get $n_4\equiv 3 ~(\mbox{mod 5})$ or $n_4\equiv 4 ~(\mbox{mod 5})$, hence $C=(1,t,t^2,t^{5x-2})$ or $C=(1,t,t^2,t^{5x-1})$, with $x\geq 1$.

If $n_3=3$, then $C=(1,t,t^3,t^{n_4})$. Arguing as before, we get $n_4\equiv 2 ~(\mbox{mod 5})$ or $n_4\equiv 4 ~(\mbox{mod 5})$, hence $C=(1,t,t^3,t^{5x+2})$ or $C=(1,t,t^3,t^{5x-1})$, with $x\geq 1$.

Assume $C=(1,t,t^2,t^{5x-2})$, with $x\geq 1$. By Proposition \ref{C}, denoting $a=F(H)-5x+2$, we get $PF(H)=\{a,a+5x-4,a+5x-3,a+5x-2\}$, with $a\geq 1$. Since the elements of $PF(H)$ have distinct nonzero values modulo 5, we get $a\equiv 1~(\mbox{mod}~5)$, hence $PF(H)=\{5n+1,5n+5x-3,5n+5x-2,5n+5x-1\}$, with $n\geq 0$. By Proposition \ref{PF}, $H=\langle 5,5n+6,5n+5x+2,5n+5x+3,5n+5x+4\rangle$. Setting $m=n+1$, we obtain $H=\langle 5,5m+1,5m+5x-3,5m+5x-2,5m+5x-1\rangle$, with $m\geq 1$. Since $H$ is of minimal multiplicity, we get $\mu(H)=5$. In the numerical semigroup $\langle 5,5m+1\rangle$, the smallest nonzero elements congruent to $2,3,4$ modulo 5 are $2(5m+1)$, $3(5m+1)$, respectively $4(5m+1)$. Then $5m+5x-3, 5m+5x-2, 5m+5x-1\notin \langle 5,5m+1\rangle$ if and only if $5m+5x-3\leq 2(5m+1)-5$, $5m+5x-2\leq 3(5m+1)-5$ and $5m+5x-1\leq 4(5m+1)-5$, which together are equivalent to $x\leq m$.

Thus $H=\langle 5,5m+1,5m+5x-3,5m+5x-2,5m+5x-1\rangle$, with $1\leq x\leq m$ is a numerical semigroup of minimal multiplicity. By Proposition \ref{PF}, $PF(H)=\{5m-4,5m+5x-8,5m+5x-7,5m+5x-6\}$, so $C=(1,t,t^2,t^{5x-2})$. This provides the first family in Table 1.

The three remaining possible modules $C$ produce the second, the third and the fourth family in Table 1, respectively.
\\

II. We will assume that $e(H)=6$.

By Theorem \ref{caract}, there exist $2\leq n_3<n_4$ such that $C=(1,t,t^{n_3},t^{n_4})$ and\\ $\{0,1,2,3,4,5\}\subseteq \{0,1,2,n_3,n_3+1,2n_3,n_4,n_4+1,n_3+n_4,2n_4\}$. Hence $C=(1,t,t^2,t^3)$ or $C=(1,t,t^2,t^4)$ or $C=(1,t,t^2,t^5)$ or $C=(1,t,t^3,t^4)$ or $C=(1,t,t^3,t^5)$.

Assume $C=(1,t,t^2,t^3)$. Then, since the elements of $PF(H)$ have distinct nonzero values modulo 6,
\begin{eqnarray}
\label{eq3} 
PF(H)&=&\{6n+1,6n+2,6n+3,6n+4\} \mbox{~or}\\
PF(H)&=&\{6n+2,6n+3,6n+4,6n+5\},\nonumber
\end{eqnarray}
with $n\geq 0$.

We will give full details for the subcase $PF(H)=\{6n+1,6n+2,6n+3,6n+4\}$ from Equation \eqref{eq3}.

By Proposition \ref{maximale}, $6n+7,6n+8,6n+9,6n+10\in \mbox{Ap}(H,e(H))$. Since the set of the smallest nonzero elements of $H$ congruent with $0,1,2,3,4,5$ modulo 6 generates $H$, we infer that there exists $m\geq 1$ such that $H=\langle 6,6n+7,6n+8,6n+9,6n+10,6m+5\rangle$ and $6m-1\notin H$.

The pseudo-Frobenius numbers of $H$ are the maximal elements of $\mathbb Z\setminus H$ with respect to the partial ordering $\leq_{H}$ (see Proposition \ref{maximale}). Hence, since $6m-1\notin H\cup PF(H)$, we get

\[\left\{
\begin{array}{l}
6n+1-(6m-1)\in H \mbox{~~or}\\
6n+2-(6m-1)\in H \mbox{~~or}\\
6n+3-(6m-1)\in H \mbox{~~or}\\
6n+4-(6m-1)\in H,
\end{array}
\right.\]
so, arguing modulo 6, we infer that

\[\left\{
\begin{array}{l}
6n+1-(6m-1)\geq 6n+8 \mbox{~~or}\\
6n+2-(6m-1)\geq 6n+9 \mbox{~~or}\\
6n+3-(6m-1)\geq 6n+10 \mbox{~~or}\\
6n+4-(6m-1)\geq 6m+5,
\end{array}
\right.\]
which is equivalent to 
\begin{equation}
\label{eq4}
n\geq 2m.
\end{equation}

Since $6n+10=2(6m+5)+6(n-2m)$, we may remove $6n+10$ from the generators of $H$, so
$$H=\langle 6,6m+5,6n+7,6n+8,6n+9\rangle.$$

Let $p\geq 0$ be an integer such that $6p+4\in \langle 6,6m+5\rangle$, which means that there exist integers $a,b\geq 0$ such that $6p+4=6a+(6m+5)b$. Arguing modulo 6, we get $b\geq 2$, so $6p+4\geq 2(6m+5)$, i.e. $p\geq 2m+1$. The latter inequality is also enough to guarantee that $6p+4\in\langle 6,6m+5\rangle$. Indeed, if $p$ is any integer with $p\geq 2m+1$, then $6p+4\geq 12m+10$. Arguing modulo 6, we get $6p+4=12m+10+6a$ for some integer $a\geq 0$. We obtain $6p+4=6a+2(6m+5)\in\langle 6,6m+5\rangle$.

We know that $6n+4\notin H$, which is equivalent to $6n+4\notin \langle 6,6m+5\rangle$. So, by the discussion above, $n\leq 2m$. Using \eqref{eq4}, we get $n=2m$.

Thus $H=\langle 6,6m+5,12m+7,12m+8,12m+9\rangle$, with $m\geq 1$.

Conversely, it is easy to check that, for such an $H$, $PF(H)=\{12m+1,12m+2,12m+3,12m+4\}$.

This finishes the subcase \eqref{eq3} and provides the fifth family in Table 1.

When $e(H)=6$, there are 5 possible modules $C$ and each has two possible sets of pseudo-Frobenius numbers modulo 6. Out of these 10 possibilities, only 8 of them produce far-flung Gorenstein numerical semigroups, namely the families numbered 5 to 12 in Table 1. It turns out that the module $(1,t,t^3,t^4)$ does not produce any far-flung Gorenstein numerical semigroup of multiplicity 6.
\\

III. We will assume that $e(H)=7$.

By Theorem \ref{caract}, there exist $2\leq n_3<n_4$ such that $C=(1,t,t^{n_3},t^{n_4})$ and\\ $\{0,1,2,3,4,5,6\}\subseteq \{0,1,2,n_3,n_3+1,2n_3,n_4,n_4+1,n_3+n_4,2n_4\}$. Hence $C=(1,t,t^2,t^3)$ or $C=(1,t,t^2,t^4)$ or $C=(1,t,t^2,t^5)$ or $C=(1,t,t^3,t^4)$ or $C=(1,t,t^3,t^5)$.

Assume $C=(1,t,t^2,t^3)$. Then $PF(H)=\{7n+1,7n+2,7n+3,7n+4\}$ or $PF(H)=\{7n+2,7n+3,7n+4,7n+5\}$ or $PF(H)=\{7n+3,7n+4,7n+5,7n+6\}$, with $n\geq 0$.

We will give full details for the subcase 
\begin{equation}
\label{eq5}
PF(H)=\{7n+1,7n+2,7n+3,7n+4\}.
\end{equation}

Considering the smallest nonzero element of $H$ in each congruence class modulo 7, there exist $m,x\geq 1$ such that $H=\langle 7,7n+8,7n+9,7n+10,7n+11,7m+5,7x+6\rangle$ and $7m-2,7x-1\notin H$. Since $7x-1\notin H\cup PF(H)$, we get

\[\left\{
\begin{array}{l}
7n+1-(7x-1)\in H \mbox{~~or}\\
7n+2-(7x-1)\in H \mbox{~~or}\\
7n+3-(7x-1)\in H \mbox{~~or}\\
7n+4-(7x-1)\in H,
\end{array}
\right.\]
so, arguing modulo 7, we infer that

\[\left\{
\begin{array}{l}
7n+1-(7x-1)\geq 7n+9 \mbox{~~or}\\
7n+2-(7x-1)\geq 7n+10 \mbox{~~or}\\
7n+3-(7x-1)\geq 7n+11 \mbox{~~or}\\
7n+4-(7x-1)\geq 7m+5,
\end{array}
\right.\]
which is equivalent to
$$n\geq m+x.$$
Then we may write $7n+11=7(n-m-x)+(7m+5)+(7x+6)$, so 
$$H=\langle 7,7n+8,7n+9,7n+10,7m+5,7x+6\rangle.$$

Let $p\geq 0$ be an integer such that $7p+4\in \langle 7,7m+5,7x+6\rangle$, which means that there exist integers $a,b,c\geq 0$ such that $7p+4=7a+(7m+5)b+(7x+6)c$. We argue modulo 7. If $c=0$, then $b\geq 5$, so $7p+4\geq 5(7m+5)$. If $b=0$, then $c\geq 3$, so $7p+4\geq 3(7x+6)$. If $b,c\geq 1$, then $7p+4\geq (7m+5)+(7x+6)$. Hence $p\geq \mbox{min}~\{5m+3,3x+2,m+x+1\}$. Conversely, it is easy to check that, for such $p$, the previous equation has solutions $a,b,c\geq 0$.

We know that $7n+4\notin H$, which is equivalent to $7n+4\notin \langle 7,7m+5,7x+6\rangle$. So, by the discussion above, $n\leq \mbox{min}~\{5m+2,3x+1,m+x\}\leq m+x$. Since $n\geq m+x$, we get $n=m+x$.

In a similar way, for the rest of the elements of $PF(H)$, we get $7n+1\notin H$ if and only if $n\leq \mbox{min}~\{3m+1,6x+4,2m+2x+2\}$, $7n+2\notin H$ if and only if $n\leq \mbox{min}~\{6m+3,5x+3,2m+x+1\}$ and $7n+3\notin H$ if and only if $n\leq \mbox{min}~\{2m,4x+2,m+2x+1\}$.

Substituting $n=m+x$ in the previous inequalities, we get

\[\left\{
\begin{array}{l}
m+x\leq 2m\\
m+x\leq 3x+1,
\end{array}
\right.\]
which is equivalent to $\frac{m-1}{2}\leq x\leq m$.

Thus $H=\langle 7,7m+5,7x+6,7m+7x+8,7m+7x+9,7m+7x+10\rangle$, with $m,x\geq 1$ and $\frac{m-1}{2}\leq x\leq m$.

Conversely, it is easy to check that, for such an $H$, $PF(H)=\{7m+7x+1,7m+7x+2,7m+7x+3,7m+7x+4\}$.

This finishes the subcase \eqref{eq5} and provides the thirteenth family in Table 1.

When $e(H)=7$, there are 5 possible modules $C$ and each has three possible sets of pseudo-Frobenius numbers modulo 7. Out of these 15 possibilities, only 11 of them produce far-flung Gorenstein numerical semigroups, namely the families 13 to 23 in Table 1.
\\

IV. We will assume that $e(H)=8$.

By Theorem \ref{caract}, there exist $2\leq n_3<n_4$ such that $C=(1,t,t^{n_3},t^{n_4})$ and\\ $\{0,1,2,3,4,5,6,7\}\subseteq \{0,1,2,n_3,n_3+1,2n_3,n_4,n_4+1,n_3+n_4,2n_4\}$. Hence\\ $C=(1,t,t^2,t^5)$ or $C=(1,t,t^3,t^4)$.

Assume $C=(1,t,t^2,t^5)$. Then $PF(H)=\{8n+1,8n+4,8n+5,8n+6\}$ or\\ $PF(H)=\{8n+2,8n+5,8n+6,8n+7\}$ or $PF(H)=\{8n+6,8n+9,8n+10,8n+11\}$ or $PF(H)=\{8n+7,8n+10,8n+11,8n+12\}$, with $n\geq 0$.

We will give some details for the subcases
\begin{equation}
\label{eq8}
PF(H)=\{8n+1,8n+4,8n+5,8n+6\}
\end{equation}
and
\begin{equation}
\label{eq9}
PF(H)=\{8n+2,8n+5,8n+6,8n+7\}.
\end{equation}

Assume $PF(H)=\{8n+1,8n+4,8n+5,8n+6\}$. Then there exist $m,x,y\geq 1$\\ such that $H=\langle 8,8n+9,8n+12,8n+13,8n+14,8m+2,8x+3,8y+7\rangle$ and\\ $8m-6,8x-5,8y-1\notin H$. Since $8m-6\notin H\cup PF(H)$, we get

\[\left\{
\begin{array}{l}
8n+1-(8m-6)\in H \mbox{~~or}\\
8n+4-(8m-6)\in H \mbox{~~or}\\
8n+5-(8m-6)\in H \mbox{~~or}\\
8n+6-(8m-6)\in H
\end{array}
\right.\]
so

\[\left\{
\begin{array}{l}
8n+1-(8m-6)\geq 8y+7 \mbox{~~or}\\
8n+4-(8m-6)\geq 8m+2 \mbox{~~or}\\
8n+5-(8m-6)\geq 8x+3 \mbox{~~or}\\
8n+6-(8m-6)\geq 8n+12
\end{array}
\right.\]
which is equivalent to
\begin{equation}
\label{eq10}
n\geq m+y \mbox{~~or~~} n\geq 2m-1 \mbox{~~or~~} n\geq m+x-1.
\end{equation}

In a similar way, since $8x-5\notin H\cup PF(H)$, we get
\begin{equation}
\label{eq11}
n\geq m+x-1 \mbox{~~or~~} n\geq 2x-1
\end{equation}
and, since $8y-1\notin H\cup PF(H)$, we get
\begin{equation}
\label{eq12}
n\geq m+y \mbox{~~or~~} n\geq 2y.
\end{equation}

Combining \eqref{eq10}, \eqref{eq11} and \eqref{eq12}, after eliminating redundant subcases, we are left with four possibilities:

i) $n\geq m+y$ and $n\geq m+x-1$

ii) $n\geq m+y$ and $n\geq 2x-1$

iii) $n\geq 2m-1$, $n\geq 2x-1$ and $n\geq 2y$

iv) $n\geq m+x-1$ and $n\geq 2y$.

We study i). We observe that $8n+9=8(n-m-y)+(8m+2)+(8y+7)$ and $8n+13=8(n-m-x+1)+(8m+2)+(8x+3)$, so $H=\langle 8,8n+12,8n+14,8m+2,8x+3,8y+7\rangle$.

Let $p\geq 0$ be an integer. One can check that $8p+1\in \langle 8,8m+2,8x+3,8y+7\rangle$ if and only if $p\geq \mbox{min}~ \{3x+1,7y+6,m+y+1,3m+x+1,x+2y+2\}$.

We know that $8n+1\notin H$, which is equivalent to $8n+1\notin \langle 8,8m+2,8x+3,8y+7\rangle$. So, by the discussion above,
\begin{equation}
\label{eq13}
n\leq \mbox{min}~ \{3x,7y+5,m+y,3m+x,x+2y+1\}\leq m+y.
\end{equation}
Since $n\geq m+y$, we get $n=m+y$.

In a similar way, $8n+5\notin H$ if and only if 
\begin{equation}
\label{eq14}
n\leq \mbox{min}~ \{7x+1,3y+1,m+x-1,3m+y,2x+y\}\leq m+x-1.
\end{equation}
Since $n\geq m+x-1$, we get $n=m+x-1$.

Since $n=m+y$ and $n=m+x-1$, we get $y=x-1$.

In a similar way, $8n+4\notin H$ if and only if 
\begin{equation}
\label{eq15}
n\leq \mbox{min}~ \{2m-1,4x,4y+2,2x+2y+1,m+x+y\}
\end{equation}
and $8n+6\notin H$ if and only if 
\begin{equation}
\label{eq16}
n\leq \mbox{min}~ \{3m-1,2x-1,2y,2m+x+y\}.
\end{equation}

Substituting $n=m+x-1$ and $y=x-1$ in the inequalities \eqref{eq13}, \eqref{eq14}, \eqref{eq15} and \eqref{eq16}, we get

\[\left\{
\begin{array}{l}
m+x-1\leq 2m-1\\
m+x-1\leq 2x-2,
\end{array}
\right.\]
a contradiction.

So there is no far-flung Gorenstein numerical semigroup satisfying condition i).

A similar discussion is made for the remaining subcases ii), iii) and iv).

We find no far-flung Gorenstein numerical semigroups satisfying \eqref{eq8}.

We now assume the subcase \eqref{eq9}: $PF(H)=\{8n+2,8n+5,8n+6,8n+7\}$. Then there exist $m,x,y\geq 1$ such that $H=\langle 8,8n+10,8n+13,8n+14,8n+15,8m+1,8x+3,8y+4\rangle$ and $8m-7,8x-5,8y-4\notin H$. Since $8m-7\notin H\cup PF(H)$, we get

\[\left\{
\begin{array}{l}
8n+2-(8m-7)\in H \mbox{~~or}\\
8n+5-(8m-7)\in H \mbox{~~or}\\
8n+6-(8m-7)\in H \mbox{~~or}\\
8n+7-(8m-7)\in H
\end{array}
\right.\]
so

\[\left\{
\begin{array}{l}
8n+2-(8m-7)\geq 8m+1 \mbox{~~or}\\
8n+5-(8m-7)\geq 8y+4 \mbox{~~or}\\
8n+6-(8m-7)\geq 8n+13 \mbox{~~or}\\
8n+7-(8m-7)\geq 8n+14
\end{array}
\right.\]
which is equivalent to
\begin{equation}
\label{eq1nou}
n\geq 2m-1 \mbox{~~or~~} n\geq m+y-1.
\end{equation}

In a similar way, since $8x-5\notin H$, we get
 \begin{equation}
\label{eq2nou}
n\geq 2x-1 \mbox{~~or~~} n\geq x+y-1
\end{equation}
and, since $8y-4\notin H$, we get
\begin{equation}
\label{eq3nou}
n\geq m+y-1 \mbox{~~or~~} n\geq x+y-1.
\end{equation}

Combining \eqref{eq1nou}, \eqref{eq2nou} and \eqref{eq3nou}, after eliminating redundant subcases, we are left with three possibilities:

i) $n\geq 2m-1$ and $n\geq x+y-1$

ii) $n\geq m+y-1$ and $n\geq 2x-1$

iii) $n\geq m+y-1$ and $n\geq x+y-1$.

We study iii). We observe that $8n+13=8(n-m-y+1)+8m+1+8y+4$ and\\ $8n+15=8(n-x-y+1)+8x+3+8y+4$, so $H=\langle 8,8n+10,8n+14,8m+1,8x+3,8y+4\rangle$.

Let $p\geq 0$ be an integer. One can check that $8p+5\in \langle 8,8m+1,8x+3,8y+4\rangle$ if and only if $p\geq \mbox{min}~ \{5m,7x+2,2m+x,m+y,3x+y+1\}$.

We know that $8n+5\notin H$, which is equivalent to $8n+5\notin \langle 8,8m+1,8x+3,8y+4\rangle$. So, by the discussion above,
\begin{equation}
\label{eq4nou}
n\leq \mbox{min}~ \{5m-1,7x+1,2m+x-1,m+y-1,3x+y\}\leq m+y-1.
\end{equation}
Since $n\geq m+y-1$, we get $n=m+y-1$.

In a similar way, $8n+7\notin H$ if and only if
\begin{equation}
\label{eq5nou}
n\leq \mbox{min}~ \{7m-1,5x,m+2x-1,4m+x-1,3m+y-1,x+y-1\}\leq x+y-1.
\end{equation}
Since $n\geq x+y-1$, we get $n=x+y-1$.

Since $n=m+y-1$ and $n=x+y-1$, we get $x=m$.

In a similar way, $8n+2\notin H$ if and only if
\begin{equation}
\label{eq6nou}
n\leq \mbox{min}~ \{2m-1,m+3x,2x+y\}
\end{equation}
and $8n+6\notin H$ if and only if
\begin{equation}
\label{eq7nou}
n\leq \mbox{min}~ \{6m-1,2x-1,3m+x-1,2m+y-1\}.
\end{equation}

Substituting $n=m+y-1$ and $x=m$ in the inequalities \eqref{eq4nou}, \eqref{eq5nou}, \eqref{eq6nou} and \eqref{eq7nou}, we get
$$m+y-1\leq 2m-1,$$
which is equivalent to $y\leq m$.

Thus $H=\langle 8,8m+1,8m+3,8y+4,8m+8y+2,8m+8y+6\rangle$, with $1\leq y\leq m$.

Conversely, it is easy to check that, for such an $H$, $PF(H)=\{8m+8y-6,8m+8y-3,8m+8y-2,8m+8y-1\}$.

This provides the family numbered 24 in Table 1.

A similar discussion is made for the remaining subcases i) and ii).

Overall, when $e(H)=8$, the two possible modules $C$ produce the families numbered 24 to 31 in Table 1.
\\

V. We will assume that $e(H)=9$.

By Theorem \ref{caract}, arguing as before, we get $C=(1,t,t^3,t^4)$.

Then $PF(H)=\{9n+1,9n+2,9n+4,9n+5\}$ or\\ $PF(H)=\{9n+2,9n+3,9n+5,9n+6\}$ or $PF(H)=\{9n+3,9n+4,9n+6,9n+7\}$ or $PF(H)=\{9n+4,9n+5,9n+7,9n+8\}$ or $PF(H)=\{9n+7,9n+8,9n+10,9n+11\}$, with $n\geq 0$.

We will give some details for the subcase 
\begin{equation}
\label{eq17}
PF(H)=\{9n+1,9n+2,9n+4,9n+5\}.
\end{equation}

It follows that there exist $m,x,y,z\geq 1$ such that $H=\langle 9,9n+10,9n+11,9n+13,9n+14,9m+3,9x+6,9y+7,9z+8\rangle$ and $9m-6,9x-3,9y-2,9z-1\notin H$. Since $9m-6\notin H\cup PF(H)$, we get

\[\left\{
\begin{array}{l}
9n+1-(9m-6)\in H \mbox{~~or}\\
9n+2-(9m-6)\in H \mbox{~~or}\\
9n+4-(9m-6)\in H \mbox{~~or}\\
9n+5-(9m-6)\in H
\end{array}
\right.\]
so

\[\left\{
\begin{array}{l}
9n+1-(9m-6)\geq 9y+7 \mbox{~~or}\\
9n+2-(9m-6)\geq 9z+8 \mbox{~~or}\\
9n+4-(9m-6)\geq 9n+10 \mbox{~~or}\\
9n+5-(9m-6)\geq 9n+11
\end{array}
\right.\]
which is equivalent to
\begin{equation}
\label{eq18}
n\geq m+y \mbox{~~or~~} n\geq m+z.
\end{equation}

In a similar way, since $9x-3\notin H\cup PF(H)$, we get
\begin{equation}
\label{eq19}
n\geq x+y \mbox{~~or~~} n\geq x+z.
\end{equation}

Since $9y-2\notin H\cup PF(H)$, we get
\begin{equation}
\label{eq20}
n\geq m+y \mbox{~~or~~} n\geq x+y \mbox{~~or~~} n\geq 2y
\end{equation}
and, since $9z-1\notin H\cup PF(H)$, we get
\begin{equation}
\label{eq21}
n\geq m+z \mbox{~~or~~} n\geq x+z.
\end{equation}

Combining \eqref{eq18}, \eqref{eq19}, \eqref{eq20} and \eqref{eq21}, after eliminating redundant subcases, we are left with three possibilities:

i) $n\geq m+y$ and $n\geq x+z$

ii) $n\geq m+z$ and $n\geq x+y$

iii) $n\geq m+z$, $n\geq x+z$ and $n\geq 2y$.

We study i). We observe that $9n+10=9(n-m-y)+(9m+3)+(9y+7)$ and $9n+14=9(n-x-z)+(9x+6)+(9z+8)$, so $H=\langle 9,9n+11,9n+13,9m+3,9x+6,9y+7,9z+8\rangle$.

Let $p\geq 0$ be an integer. One can check that $9p+1\in \langle 9,9m+3,9x+6,9y+7,9z+8\rangle$ if and only if $p\geq \mbox{min}~ \{4y+3,8z+7,m+y+1,m+2z+2,2x+y+2,2x+2z+3,y+6z+6,x+2y+z+3\}$.

We know that $9n+1\notin H$, which is equivalent to $9n+1\notin \langle 9,9m+3,9x+6,9y+7,9z+8\rangle$. So, by the discussion above, $9n+1\notin H$ if and only if
\begin{equation}
\label{eq22}
n\leq \mbox{min}~ \{4y+2,8z+6,m+y,m+2z+1,2x+y+1,2x+2z+2,y+6z+5,x+2y+z+2\}\leq m+y.
\end{equation}
Since $n\geq m+y$, we get $n=m+y$.

In a similar way, $9n+5\notin H$ if and only if 
\begin{equation}
\label{eq23}
n\leq \mbox{min}~ \{2y,4z+2,2m+z,x+z,y+2z+1\}\leq x+z.
\end{equation}
Since $n\geq x+z$, 
we get $n=x+z$.

Since $n=m+y$ and $n=x+z$, we get $z=m+y-x$.

In a similar way, $9n+2\notin H$ if and only if 
\begin{equation}
\label{eq24}
n\leq \mbox{min}~ \{8y+5,7z+5,2m+2y+1,m+z,x+2y+1,2x+z+1,3y+z+2,x+y+2z+2\}
\end{equation}
and $9n+4\notin H$ if and only if 
\begin{equation}
\label{eq25}
n\leq \mbox{min}~ \{7y+4,5z+3,2m+y,2m+2z+1,x+y,x+2z+1,2y+z+1\}.
\end{equation}

Substituting $n=m+y$ and $z=m+y-x$ in the inequalities \eqref{eq22}, \eqref{eq23}, \eqref{eq24} and \eqref{eq25}, we get

\[\left\{
\begin{array}{l}
m+y\leq 2y\\
m+y\leq 4m+4y-4x+2\\
m+y\leq 2m+y-x\\
m+y\leq x+y\\
m+y\leq m+3y-x+1\\
m+y\leq 2m+3y-2x+1,
\end{array}
\right.\]
which is equivalent to 

\[\left\{
\begin{array}{l}
x=m\\
y\geq m.
\end{array}
\right.\]
Hence $z=m+y-x=y$.

As $9n+11=(9m+3)+(9z+8)$ and $9n+13=(9x+6)+(9y+7)$, we obtain $H=\langle 9,9m+3,9m+6,9y+7,9y+8\rangle$, with $1\leq m\leq y$.

Conversely, it is easy to check that, for such an $H$, $PF(H)=\{9m+9y+1,9m+9y+2,9m+9y+4,9m+9y+5\}$.

A similar discussion is made for the remaining subcases.

Overall, when $e(H)=9$, the module $(1,t,t^3,t^4)$ produces the families 32 to 40 in Table 1.
\end{proof}

\medskip

The given generators of the numerical semigroups in Table 1 are usually minimal. It is easy to check this fact for the families of minimal multiplicity (no. 1--4), for those which depend on one parameter (no. 5--12, 15--17, 21) and also for families 32,39 and 40. For the rest, after a case by case analysis, we summarize in Table 2 for which $x$ and $m$ the generators are not minimal and also which generators can be discarded.

\begin{table}[h]
\caption{The families in Table 1 which are not minimally generated}
{\scriptsize
\begin{tabular}{|c|c|c|}
\hline
No. & Not minimally generated for & Non-minimal generators\\ \hline
\multirow{2}{*}{13} & $x=\frac{m-1}{2}$ & $7m+5$\\
\cline{2-3}
{} & $x=m$ & $7m+7x+10$\\ \hline
\multirow{2}{*}{14} & $x=m$ & $7m+7x+4$\\
\cline{2-3}
{} & $x=2m$ & $7x+2$\\ \hline
\multirow{2}{*}{18} & $x=m-1$ & $7m+7x+5$\\
\cline{2-3}
{} & $x=m$ & $7m+7x+8$\\ \hline
\multirow{2}{*}{19} & $x=m$ & $7m+7x+12$\\
\cline{2-3}
{} & $x=2m$ & $7x+6$\\ \hline
\multirow{2}{*}{20} & $x=\frac{m-1}{2}$ & $7m+1$\\
\cline{2-3}
{} & $x=m$ & $7m+7x+2$\\ \hline
\multirow{2}{*}{22} & $x=m-1$ & $7m+7x+1$\\
\cline{2-3}
{} & $x=2m$ & $7x+4$\\ \hline
\multirow{2}{*}{23} & $x=\frac{m-1}{2}$ & $7m+3$\\
\cline{2-3}
{} & $x=m$ & $7m+7x+6$\\ \hline
24 & $x=m$ & $8m+8x+2,8m+8x+6$\\ \hline
25 & $x=m$ & $8m+8x+10$\\ \hline
\multirow{2}{*}{26} & $x=m$ & $8m+8x+12,8m+8x+13$\\
\cline{2-3}
{} & $x=2m$ & $8x+6$\\ \hline
27 & $x=\frac{m-1}{2}$ & $8m+6$\\ \hline
28 & $x=m-1$ & $8m+8x+6$\\ \hline
\multirow{2}{*}{29} & $x=m$ & $8m+8x+4,8m+8x+7$\\
\cline{2-3}
{} & $x=2m$ & $8x+2$\\ \hline
30 & $x=\frac{m-1}{2}$ & $8m+2$\\ \hline
31 & $x=m$ & $8m+8x+6,8m+8x+10$\\ \hline
33 & $x=\frac{2m-1}{3}$ & $9m+6,9x+7$\\ \hline
\multirow{2}{*}{34} & $x=\frac{2m-2}{3}$ & $18m-9x-2$\\
\cline{2-3}
{} & $x=m-1$ & $18m+5$\\ \hline
\multirow{2}{*}{35} & $x=\frac{m-1}{2}$ & $9m-1$\\
\cline{2-3}
{} & $x=m-1$ & $9m+9x+6$\\ \hline
36 & $x=\frac{3m-1}{2}$ & $18x-9m+11$\\ \hline
37 & $x=\frac{m-1}{2}$ & $9m+1$\\ \hline
\multirow{2}{*}{38} & $x=m$ & $18x+5,18x+8$\\
\cline{2-3}
{} & $x=2m$ & $9x+2,18x-9m+3$\\ \hline
\end{tabular}}
\end{table}

\newpage

{\bf Acknowledgement.} We gratefully acknowledge the use of the numericalsgps package (\cite{package}) in GAP (\cite{GAP}) for our computations.

We would like to thank Dumitru Stamate and Mihai Cipu for telling us about far-flung Gorenstein rings and for suggesting this problem to us.

We are grateful to the anonymous referee for reading this paper carefully and giving useful comments.

The author was partly supported by a grant of the Ministry of Research, Innovation and Digitization, CNCS - UEFISCDI, project number PN-III-P1-1.1-TE-2021-1633, within PNCDI III.

This work is part of the author's PhD thesis at the University of Bucharest.

{}

\end{document}